\documentclass[letterpaper,11pt,reqno]{amsart}
\usepackage{amssymb, amsmath, amsfonts, amscd}
\usepackage{mathrsfs, latexsym}
\usepackage[all,ps,cmtip]{xy}
\usepackage{array, longtable}
\usepackage{bm, enumitem}
\usepackage{xcolor}
\usepackage{tikz}
\usepackage{tikz-cd}
\usepackage{comment}

\title{On minimal 3-folds with  $K^3\geq 86$}
\date{\today} 

\author{Meng Chen}
\address{\rm School of Mathematical Sciences \& Shanghai Center for Mathematical Sciences, Fudan University, Shanghai 200433, China}
\email{mchen@fudan.edu.cn}

\author{Sicheng Ding}
\address{\rm School of Mathematical Sciences, Fudan University, Shanghai 200433, China}
\email{22110180008@m.fudan.edu.cn}

\thanks{This work was supported by NSFC for Innovative Research Groups \#12121001 and National Key Research and Development Program of China  \#2020YFA0713200}
\newcommand{\bC}{{\mathbb C}}
\newcommand{\bQ}{{\mathbb Q}}
\newcommand{\bP}{{\mathbb P}}
\newcommand{\roundup}[1]{\lceil{#1}\rceil}
\newcommand{\rounddown}[1]{\lfloor{#1}\rfloor}

\newcommand\Pic{\text{\rm Pic}}
\newcommand\Vol{\text{\rm Vol}}
\newcommand\lrw{\longrightarrow}

\newcommand\OO{{\mathcal{O}}}
\newcommand\OX{{\mathcal{O}_X}}

\newcommand{\lsgeq}{\succcurlyeq}

\newcommand{\Mov}{\operatorname{Mov}}

\newtheorem{thm}{Theorem}[section]
\newtheorem{lem}[thm]{Lemma}
\newtheorem{cor}[thm]{Corollary}
\newtheorem{prop}[thm]{Proposition}

\newtheorem{op}[thm]{Problem}

\theoremstyle{definition}

\newtheorem{question}[thm]{Question}
\newtheorem{exmp}[thm]{Example}

\newtheorem{rem}[thm]{Remark}

\theoremstyle{remark}

\begin{document}
\begin{abstract} We prove that the $5$-canonical map of every minimal projective $3$-fold $X$ with $K_X^3\geq 86$ is stably birational onto its image, which loosens previous requirements $K_X^3>4355^3$ and $K_X^3>12^3$ respectively given by Todorov and Chen. The essential technical ingredient of this paper is an efficient utilization of a moving divisor which grows from global $2$-forms by virtue of Chen-Jiang.  
\end{abstract}


\maketitle

\pagestyle{myheadings}
\markboth{\hfill M. Chen, S. Ding\hfill}{\hfill Minimal 3-folds with $K^3\geq 86$ \hfill}
\numberwithin{equation}{section}


\section{Introduction} 

We work over the complex number field $\bC$. In birational geometry, it is important to study the geometry of pluricanonical linear systems $|mK|$. Let $X$ be a normal projective variety of dimension $n$. For a positive integer $m$, we denote by $\varphi_m=\varphi_{m,X}$ the $m$-canonical map of $X$. Hacon-McKernan \cite{HM06}, Takayama \cite{Takayama06} and Tsuji \cite{Tsuji06} independently proved that there exists an optimal constant $r_n$, depending only on $n$, such that $\varphi_m$ is birational for all $m\geq r_n$ and  for all nonsingular projective varieties of general type. It is well-known that $r_1=3$, $r_2=5$ by Bombieri \cite{Bombieri73} and $r_3\leq 57$ by Chen-Chen \cite{EXPIII,Chen18}.

Some optimal birationality indices $\tilde{r}_3$ for restricted $3$-folds of general type were obtained in Chen \cite{Chen03} ($\tilde{r}_3=8$ for those with $p_g\geq 2$), Chen-Zuo \cite{CZuo08} ($\tilde{r}_3=14$ for those with $\chi(\OO)\leq 0$), 
Chen-Chen-Zhang \cite{CCZ07} ($\tilde{r}_3=5$ for Gorenstein and minimal objects),  
 Chen-Chen-Chen-Jiang \cite{CCCJ} ($\tilde{r}_3=5$ in irregular case, i.e. $q(X)>0$). Recently, 
 Chen-Jiang \cite{CJ22} showed that 
one has $\Tilde{r}_3=3$
 for all 3-folds of general type with $h^2(\OX)\geq108\cdot18^3+4$.
For projective varieties with large canonical volumes, the \textit{lifting principle} holds in the following sense: there exists a constant $N_n$ depending only on the dimension $n>1$ such that, for all $n$-folds $X$ with the canonical volume $\geq N_n$, $\varphi_{m,X}$ is birational for $m\geq r_{n-1}$. Roughly speaking, for varieties of general type with large canonical volumes, one has $\Tilde{r}_n=r_{n-1}$. Such principle was first proved in the surface case with $N_2=3$ by Bombieri \cite{Bombieri73}. Then it was extended to the case $n=3$ by Todorov \cite{Todorov}, to cases $n=4,5$ by Chen-Jiang \cite{CJ17}, and to rest cases $n\geq 6$ by Chen-Liu \cite{CL24} (see also Wang \cite{W25}).

\begin{question} Let $n\geq3$ be an integer. Find the optimal constant $N_n$ such that, for all nonsingular projective $n$-folds $Y$ of general type with $\Vol(Y)\geq N_n$, $\varphi_{m,Y}$ is birational for all $m\geq r_{n-1}$.
\end{question}

By utilizing the effective induction on non-klt centers developed in Hacon-McKernan \cite{HM06} and Takayama \cite{Takayama06}, Todorov \cite{Todorov} showed  $N_3\leq4355^3$. Improvements of Di Biagio \cite{DiBiagio12} and Chen \cite{Chen12} were, respectively, $N_3\leq2\times1917^3$ and $N_3\leq12^3$. 

The purpose of this article is to present a more practical upper bound of $N_3$ as follows.   

\begin{thm}\label{mainthm}
    Let $X$ be a minimal projective 3-fold with the canonical volume $\Vol(X)\geq 86$. Then $|mK_X|$ gives a birational map for all $m\geq 5$.
\end{thm}

We now briefly explain the proof of Theorem \ref{mainthm}. The effective induction on non-klt centers used in previous studies \cite{Chen12,DiBiagio12,HM06,Takayama06,Todorov} is powerful for showing the existence of such a constant $N_3$, but it is still far from obtaining the optimality of $N_3$,  especially when non-klt centers have low codimensions and small canonical volumes. Although the application of McKernan's covering family of tigers in \cite{McKernan02} by Todorov \cite{Todorov} and the canonical restriction inequality by Chen \cite{Chen14} are more and more effective in relaxing the bound of  $N_3$, further improvement along these lines appears extremely difficult. It is therefore natural to consider using Tankeev's principle \cite[Section 2.1]{Che01}, which is effective both for showing birationality and for obtaining a satisfactory volume bound. The primary obstacle is that we need the nonemptyness of either $|K_X|$ or $|2K_X|$ in order to apply Tankeev's principle to show the birationality of $\varphi_5$.

Fortunately, by virtue of the work of Chen-Jiang \cite{CJ22}, for those $X$ with many global 2-forms, we may construct a fibration using another divisor $L$ instead of $2K_X$ (see Set-up \ref{assumption}), where $2K_X-L$ is pseudo-effective. Otherwise, by Reid's Riemann-Roch formula, we may assume that $P_2(X)=h^0(2K_X)$ is sufficiently large. In both cases we can  construct a fibration from a birational model $X'$ of $X$ onto some normal variety $\Sigma$ and use Tankeev's principle. When $\dim\Sigma=1$, we have a pencil of surfaces on $X'$ and may deduce the birationality of $\varphi_5$. When $\dim\Sigma=2$ or 3, we need to show that $|5K_{X'}|$ can separate the generic irreducible elements of $|L|$, say $|M'|$, then show that the restriction of $|5K_{X'}|$ to the generic irreducible element $S\in|M'|$ is birational, which needs to overcome  additional difficulties.

To show that $|5K_{X'}|$ separates generic irreducible elements of $|M'|$, it suffices to show that    $|5K_{X'}-M'|$ is non-empty. A natural approach to establish the non-vanishing of such an adjoint linear system is to apply the effective induction on non-klt centers developed by Hacon-McKernan \cite{HM06}. The problem is that, when non-klt centers have low codimensions and small canonical volumes, the required lower bound on the canonical volume of $X$ is much larger than what we expect. Todorov \cite{Todorov} deals with this by constructing a covering family of tigers, which is a pencil of surfaces on a finite cover of $X'$. But, in our situation, we only need to consider regular $X$ (i.e. $q(X)=0$) since $\varphi_5$ is known to be birational for irregular $X$ by Chen-Chen-Chen-Jiang \cite{CCCJ}, and, in this case, we may descend this pencil of surfaces to $X'$. Then we reduce to a more manageable case.

 Proving the birationality of $|5K_{X'}||_S$  is just an application of a Reider-type theorem due to Masek and Langer (see \cite[Theorem 0.1]{Masek1999} or \cite{Langer2001}). The most difficult case is $\dim\Sigma=2$ and the general fiber of $X'\to\Sigma$ is a curve of genus two with small canonical degree, and, in this case, we do not have sufficient positivity to assert the birationality of $|5K_{X'}|$. So we consider an alternative approach. Being inspired by Chen-Zhang \cite{CZhang08} and Chen-Hu \cite{CH21}, we believe that a 3-fold of general type, with large birational invariants and admitting a genus 2 fibration structure whose general fiber has small canonical degree, should be birationally fibred by (1,2)-surfaces. However, the strategy in \cite{CH21} does not work when $p_g(X)=0$. So, once again, we employ Takayama's effective induction on non-klt centers \cite{Takayama06} to show the existence of a $(1,2)$-surface fibration. Finally this allows us to obtain the birationality of $|5K_{X'}|$. This step requires $h^0(L)\geq 5$, which is exactly the reason behind the threshold $K_X^3 \geq 86$.

The structure of this article is as follows. In Section 2, we introduce some basic notations and lemmas. Sections 3-5 devote to proving  Theorem \ref{mainthm}.

\section{Preliminaries}\label{sec: 2}

\subsection{Notations and conventions}
In this article, a \textit{variety} means an integral separated scheme of finite type over $\bC$. We always work on normal projective varieties. A \textit{divisor} is referred to a Weil divisor, and a $\bQ$\textit{-divisor} is referred to a $\bQ$-linear combination of Weil divisors. We denote by $\sim$ the linear equivalence, $\sim_{\bQ}$ the $\bQ$-linear equivalence, and $\equiv$ the numerical equivalence. For two $\bQ$-divisors $A$ and $B$, we write $A\geq B$ if $A-B$ is $\bQ$-linear equivalent to an effective $\bQ$-divisor. For two linear systems $|D_1|$ and $|D_2|$, we write $|D_1|\lsgeq|D_2|$ if $|D_1|\supseteq|D_2|+\text{certain fixed effective divisor}$. We refer to \cite{HM06} for those standard notions of \textit{volume} and \textit{(minimal) non-klt center}. A surface is called a \textit{(1,2)-surface} if it is a nonsingular projective surface whose minimal model satisfies: $c_1^2=1$ and $p_g=2$.
A 3-fold is said to have \textit{QFT singularities} if it is $\bQ$-factorial and has terminal singularities. A projective 3-fold $X$ is called \textit{minimal} if $X$ has QFT singularities and $K_X$ is nef.

\subsection{Technical preparation}\

{}We need the following lemma to prove the birationality of $\varphi_5$ when $X$ admits a pencil. It is an application of a Reider-type theorem on adjoint linear systems on surfaces.
\begin{lem}\label{5kfib}
    Let $X$ be a minimal projective 3-fold of general type, $\pi:X'\to X$ a resolution of $X$, and $f:X'\to \bP^1$ a fibration. Denote by $F$ a general fiber of $f$. Suppose $\pi^*(K_X)\sim_{\bQ} pF+E_p$ for some rational number $p$ and an effective $\bQ$-divisor $E_p$. Let $m\geq 5$ be an integer. Then $|mK_X|$ gives a birational map under one of the following conditions: 
    \begin{enumerate}
        \item  $p>4$;
        \item $p>(5\sqrt{2}+6)/4$ and $F$ is not an $(1,2)$-surface;
        \item $p>3/2$, $p_g(F)>0$, and $\Vol(F)\geq 2$.
        \item $p>8/7$, $p_g(F)>0$, and $\Vol(F)\geq 3$.
        \item $p>1$, $p_g(F)>0$, and $\Vol(F)\geq 4$.
    \end{enumerate}
\end{lem}
\begin{proof}
We follow the idea of \cite[Proposition 2.1]{Chen12}. By \cite[Lemma 3.3]{CZuo08}, 
there exists an effective $\bQ$-divisor $H_p$ on $F$ such that 
 \begin{align}\label{canonicalres2.1}
     \pi^*(K_X)|_F\equiv\frac{p}{p+1}\sigma^*(K_{F_0})+H_p,
 \end{align}
 where $\sigma:F\to F_0$ is the contraction onto the minimal model. 

 Pick two distinct general fibers $F_1,F_2$ of $f$. Both $F_1$ and $F_2$ are known to be nonsingular projective surfaces of general type. Since 
 $$(m-1)\pi^*(K_{X})-F_1-F_2-\frac{2}{p}E_p\equiv(m-1-\frac{2}{p})\pi^*(K_X)$$
 is nef and big, Kawamata-Viehweg vanishing theorem \cite{KV,VV} gives the surjective map 
 \begin{eqnarray}\label{2.1disfib}
     &&H^0(X',K_{X'}+\ulcorner(m-1)\pi^*(K_{X})-\frac{2}{p}E_p\urcorner) \notag \\& \rightarrow &
\bigoplus_{i=1}^2H^0(F_i,K_{F_i}+\ulcorner(m-1)\pi^*(K_{X})-F_i-\frac{2}{p}E_p\urcorner|_{F_i}).
 \end{eqnarray}
 Note that
 \begin{align*}
     \ulcorner(m-1)\pi^*(K_{X})-F_i-\frac{2}{p}E_p\urcorner|_{F_i}\geq l\sigma_i^*(K_{F_{i,0}})+\ulcorner N_l\urcorner,
 \end{align*}
 where $l=1,2$, $\sigma_i:F_i\rightarrow F_{i,0}$ is the contraction onto the minimal model for $i=1,2$ and 
 \begin{align*}
     N_l&=((m-1)\pi^*(K_{X})-F_i-\frac{2}{p}E_p)|_{F_i}-l\sigma_i^*(K_{F_{i,0}})-\dfrac{l(p+1)}{p}H_p\\
     &\equiv (m-l-1-\frac{l+2}{p})\pi^*(K_X)|_{F_i}
 \end{align*}
 is nef and big whenever $p>2$ and $l=2$ (resp. $p>1$ and $l=1$). Therefore, under any of Conditions (1)–(5), the right groups in (\ref{2.1disfib}) is non-trivial by \cite[Lemma 2.14]{EXPII}. This implies that $$|K_{X'}+\ulcorner(m-1)\pi^*(K_{X})-\frac{2}{p}E_p\urcorner|$$ separates different general fibers of $f$ and so does $|mK_{X'}|$.

 Similarly, we also have
 \begin{eqnarray*}
    &&|K_{X'}+\ulcorner(m-1)\pi^*(K_{X})-\frac{1}{p}E_p\urcorner||_{F} \\&=&|K_{F}+\ulcorner(m-1)\pi^*(K_{X})-F-\frac{1}{p}E_p\urcorner|_{F}|\\
    &\lsgeq& |K_F+\ulcorner Q_m\urcorner|,
 \end{eqnarray*}
 where
 \begin{equation*}
     Q_m=((m-1)\pi^*(K_{X})-F-\frac{1}{p}E_p)|_{F}\equiv(m-1-\frac{1}{p})\pi^*(K_{X})|_{F}
 \end{equation*}
is nef and big. So it is sufficient to show that $|K_F+\ulcorner Q_m\urcorner|$ gives a birational map by Tankeev's principle \cite[Section 2.1]{Che01}. Note that under any of Conditions (1)–(5), we have
\begin{align*}
    Q_m^2&=(m-1-\frac{1}{p})^2(\pi^*(K_{X})|_{F})^2\geq\frac{p^2}{(p+1)^2}(m-1-\frac{1}{p})^2(\sigma^*(K_{F_0}))^2\\
    &\geq\frac{(4p-1)^2}{(p+1)^2}\Vol(F)>8.
\end{align*}
Suppose $|K_F+\ulcorner Q_m\urcorner|$ does not give a birational map, then by Masek-Langer's criterion (see \cite[Theorem 0.1]{Masek1999} or \cite{Langer2001}), there exists a family of curves $C$ passing through two very general points such that 
\begin{equation}\label{masekpf1}
    (Q_m\cdot C)\leq\frac{4}{1+\sqrt{1-\frac{8}{Q_m^2}}}\leq \frac{4}{1+\sqrt{1-\frac{8(p+1)^2}{(4p-1)^2\Vol(F)}}}.
\end{equation}
On the other hand, we have 
\begin{eqnarray}
    (Q_m\cdot C)&=& (m-1-\frac{1}{p})(\pi^*(K_{X})|_{F}\cdot C)\notag\\
    &\geq &\frac{4p-1}{p+1}(\sigma^*(K_{F_0})\cdot C)\geq \frac{4p-1}{p+1}.\label{2.4}
\end{eqnarray}

When $p>4$, Relations \eqref{masekpf1} and \eqref{2.4} are contradictory to each other. So we conclude Statement (1). 

When $F$ is not an (1,2)-surface, we have $(\sigma^*(K_{F_0})\cdot C)\geq 2$ by \cite[Lemma 2.5]{EXPIII}. Then \eqref{2.4} reads 
\begin{equation}(Q_m\cdot C)\geq \frac{8p-2}{p+1}.\label{2.5}\end{equation}
 Relations \eqref{masekpf1} and \eqref{2.5} are always contradictory to each other under any of Conditions (2)–(5). So we conclude other statements.
\end{proof}

In order to deal with the case that non-klt centers have small codimensions, Todorov \cite[Lemma 3.2]{Todorov} considered a covering family of tigers. In fact we can work out this case more efficiently on regular threefolds:

\begin{lem}\label{pencil}
Let $X$ be a nonsingular projective 3-fold with $h^1(X,\OX)=0$, $X’$ be another nonsingular projective 3-fold. Let $g:X'\to X$ be a generically finite morphism, and $h:X'\to B$ be a fibration onto a nonsingular curve $B$. Denote by $X_b'$ the fiber over $b\in B$. Then, for general points $b_1,b_2\in B$, $g(X_{b_1}')\sim g(X_{b_2}')$. 
\end{lem}

\begin{proof} By assumption, we have the following diagram:
$$
\begin{tikzcd}
X' \arrow[r, "g"] \arrow[d, "h"] & X \\
B                                 &   
\end{tikzcd}
$$

    Consider the graph $\Gamma_h=\{(x',h(x'))\in X'\times B\mid x'\in X'\}$, which, as a subvariety of $X'\times B$, is flat over $B$, and gives an algebraic family of divisors $\{X_b'\}_{b\in B}$ parametrized by $B$. Let $\Gamma$ be the closure of the image of $\Gamma_h$ under the morphism $g\times \textrm{id}:X'\times B\to X\times B$. Then $\Gamma$ also gives an algebraic family of divisors $\{g(X_b')\}_{b\in B}$ on $X$ parametrized by $B$ by \cite[Chapter III, Proposition 9.7 and Example 9.8.5]{Har}. In other words, $g(X_{b_1}')$ and $g(X_{b_2}')$ are algebraically equivalent by the definition of N{\'e}ron-Severi group. But by the exponential exact sequence, the natural map $c_1:\Pic(X)\to NS(X)$ is an isomorphism. It follows that $g(X_{b_1}')$ and $g(X_{b_2}')$ are linearly equivalent.
\end{proof}

When $2K_X$ does not have enough global sections, we need to consider another numerically related sub-divisor $L$ instead. The following lemma, growing from the proof of \cite[Theorem 2.1]{CJ22}, will be used in this paper.

\begin{lem}\label{2forms}
    Let $X$ be a minimal projective $3$-fold of general type. Assume $h^2(\OX)\geq 13$. Then there exists a Weil divisor $L$ on $X$ such that $h^0(X,L)\geq 5$ and that $D=2K_X-L$ is pseudo-effective.
\end{lem}
\begin{proof}
    Let $\pi:X'\to X$ be a resolution of $X$. Since $X$ has rational singularity, we know $h^0(\Omega_{X'}^2)= h^2(\mathcal{O}_{X'})=h^2(\OX) \geq 13 $. By the proof of \cite[Theorem 2.1]{CJ22}, there exists a divisor $L'$ on $X'$ such that $h^0(X',L')>4$, and $D'=2K_{X'}-L'$ is pseudo-effective. Let $L$ and $D$ be the birational pushforward of $L'$ and $D'$ respectively. Then $h^0(X,L)\geq h^0(X',L')>4$ and $D=2K_X-L$ is pseudo-effective.
\end{proof}

\section{The birationality of $\varphi_{5,X}$ (Part I: $d_L=1$ and $3$)}\label{sec: 3}

\subsection{Set-up for  $\varphi_{L,X}$}\label{assumption}\

Let $X$ be a minimal projective $3$-fold of general type and $L$ a Weil divisor on $X$ with $h^0(X,L)\geq 2$. 
Let $\varphi_{L}:X\dashrightarrow\bP^{h^0(L)-1}$ be the rational map induced by the complete linear system $|L|$. By Hironaka's big theorem, we may take a birational morphism $\pi:X'\to X$ 
such that $X'$ is nonsingular, and the movable part of $|\rounddown{\pi^*L}|$, say $|M'|$, is base point free. Then $\varphi_{M'}=\varphi_L\circ\pi:X'\to\overline{\varphi_L(X)}$ is a morphism.  We may write
\begin{align*}
    K_{X'}=\pi^*(K_X)+E_{\pi}\  \text{and}\  \pi^*L\sim_{\bQ} M'+Z',
\end{align*}
where $E_{\pi}$ is a $\pi$-exceptional divisor and $Z'$ is an effective $\bQ$-divisor. Take the Stein factorization of $\varphi_{M'}$, say $X^{'}\stackrel{f}\rightarrow \Sigma \stackrel{s} \rightarrow \overline{\varphi_L(X)}$. We have the following commutative diagram:
$$\begin{tikzcd}
X' \arrow[dd, "\pi"'] \arrow[rr, "f"] \arrow[rrdd, "\varphi_{M'}"] &  & \Sigma \arrow[dd, "s"] \\
     &  &       \\
X \arrow[rr, "\varphi_L", dashed]                                  &  & \overline{\varphi_L(X)}         
\end{tikzcd}$$
Denote by $d_L$ the dimension of $\Sigma$, and $S$ a generic irreducible element of $|M'|$ (see \cite[Section 2.4]{Chen03} for the definition). Next, we prove the following proposition, which is a reduction of Theorem \ref{mainthm}.

\subsection{The case with $d_L=1$}\
\begin{prop}\label{d=1} Let $m\geq 5$ be an integer and $X$ a minimal projective $3$-fold of general type. Assume that $X$ satisfies the following conditions:
\begin{enumerate}
    \item $\Vol(X)>\alpha_0^3>12$ for some number $\alpha_0>0$;
    \item $q(X)=h^1(X,\OX)=0$;
    \item $\chi(\OX)\neq1$;
\item there is a Weil divisor $L$ on $X$ such that $h^0(X,L)\geq 5$ and that 
$D=2K_X-L$ is pseudo-effective.
\end{enumerate}
If $d_L=1$, then $|mK_X|$ gives a birational map.
\end{prop}
\begin{proof} Since $q(X)=0$, $|L|$ induces the fibration $f:X'\lrw \bP^1$. Denote by $F$ a general fiber of $f$. If $p_g(F)=0$, then $q(F)=0$ and $p_g(X)=0$, which implies $\chi(\OO_X)=1$, contradicting to the assumption. Hence we have $p_g(F)>0$. 

Suppose $\Vol(F)\geq 2$. We have $M'\sim aF$, where $a\geq h^0(L)-1\geq 4$. It follows that $\pi^*(K_X)\sim_{\bQ}(3/2+\delta)F+\Delta$, where $\delta$ is a sufficiently small positive rational number and $\Delta$ is an effective $\bQ$-divisor. Then the statement follows from Lemma \ref{5kfib}(3).  Suppose $\Vol(F)=1$. Since $\Vol(X)>12$. Then we have $\pi^*(K_{X})\geq (4+\delta)F$ by \cite[Lemma 2.6(1)]{Chen14} for some small positive rational number $\delta$. The statement follows directly from Lemma \ref{5kfib}(1).
\end{proof}

\subsection{The non-vanishing of $H^0(X',5K_{X'}-M')$ when $d_L\geq 2$}\

{}We want to apply Tankeev's principle to prove the birationality of $|mK_{X'}|$ ($m\geq 5$) when $d_L\geq 2$. A necessary step is to show that $|mK_{X'}|$ can separate different general members of $|M'|$.  Hence it suffices to show $|5K_{X'}-M'|\neq\varnothing$ when $d_L\geq 2$. The proof of the following proposition follows the idea of \cite[Section 3]{Todorov} except that the last step is implemented by an alternative argument.

\begin{prop}\label{tankeev1} Let $m\geq 5$ be an integer and $X$ a minimal projective $3$-fold of general type. Assume that $X$ satisfies the following conditions:
\begin{enumerate}
    \item $\Vol(X)>\alpha_0^3>36$ for some number $\alpha_0>0$;
    \item $q(X)=h^1(X,\OX)=0$;
    \item $\chi(\OX)\neq1$;
\item there is a Weil divisor $L$ on $X$ such that $h^0(X,L)\geq 3$ and that 
$D=2K_X-L$ is pseudo-effective.
\end{enumerate}
Suppose that $\varphi_{m,X}$ is non-birational and that $d_L\geq 2$. Then 
 $|mK_{X'}-M'|\neq \varnothing$.
\end{prop}
\begin{proof} 
We only consider the case with $m=5$. The method works for other situations with $m>5$.

    Suppose that $\varphi_{5,X}$ is non-birational. We want to show that $|5K_{X'}-M'|\neq \varnothing$. Then, by \cite[Lemma 2.6]{HM06}, it suffices to show that there exists some rational number $0<\lambda'<1$, as well as some effective $\bQ$-divisor $\Delta'\sim_{\bQ}\lambda'(4K_{X'}-M')$ such that $(X,\Delta')$ has an isolated non-klt center of dimension zero at some very general point $x$. 
    
    By construction, $4K_{X'}-M'\sim_{\bQ}2K_{X'}+\Delta_0$, where $\Delta_0=2E_{\pi}+Z'+\pi^*D$ is a pseudo-effective $\bQ$-divisor. Since $K_{X'}$ is big, we may fix a sufficiently large integer $n$ and an effective $\bQ$-divisor $\Delta_n\sim_{\bQ}\Delta_0+\frac{1}{n}K_{X'}$. We may choose $x\in X'$ to be a smooth point at very general position and not in the support of $\Delta_n$, then it suffices to construct a very singular divisor proportional to $(2-1/n)K_{X'}$.

    Since $\Vol(X)>\alpha_0^3$, By \cite[Lemma 10.4.12]{La}, there exist some $\lambda<3/\alpha_0$ and $\Delta_x\sim_{\bQ}\lambda K_{X'}$ such that $(X',\Delta_x)$ is not log canonical at $x$. After perturbation (see \cite[Lemma 2.5]{HM06}) and possibly replacing $\lambda$, we may assume there exist some $\lambda<3/\alpha_0$ and $\Delta_x\sim_{\bQ}\lambda K_{X'}$ such that there exists a unique non-klt center of $(X',\Delta_x)$ at $x$, say $V_x$. We next discuss case by case according to the dimension of $V_x$.
\medskip

    \textbf{Case I} - $\dim V_x=0$. 
    
    In this case, $x$ is an isolated non-klt center of $(X',\Delta_x)$, and also of $(X',\Delta_x+\frac{n\lambda}{2n-1}\Delta_n)$. Since
    $$\Delta_x+\frac{n\lambda}{2n-1}\Delta_n\sim_{\bQ} \frac{n\lambda}{2n-1}(4K_{X'}-M'),$$ 
    to apply \cite[Lemma 2.6]{HM06}, we need $n\lambda/(2n-1)<1$. This can be achieved for $\alpha_0>3/2$ and for a sufficiently large number $n$.
\medskip

    \textbf{Case II} - $\dim V_x=1$. 
     
    We hope to cut down the dimension of the non-klt center to reduce to Case I. Let $W_x$ be the normalization of $V_x$. Since $x$ is a very general point, we may assume that $W_x$ is a curve of general type. Hence $\Vol(W_x)\geq 2>2-1/n$. By \cite[Lemma 10.4.12]{La}, there exists some number $\mu<n/(2n-1)$ and $\Theta_x\sim_{\bQ}\mu K_{W_x}$ such that $\Theta_x$ has multiplicity larger than 1 at some point on $W_x$. Then, by \cite[Theorem 4.1]{HM06}, there exists some effective $\bQ$-divisor $\Delta_x^{(1)}\sim_{\bQ}(\lambda+\mu+\lambda\mu+1/n)K_{X'}$ such that $(X',\Delta_x^{(1)})$ has an isolated non-klt center of dimension zero at some very general point on $X'$. As in the previous case, replacing $\Delta_x$ with $\Delta_x^{(1)}$, we need 
     $n(\lambda+\mu+\lambda\mu+1/n)/(2n-1)<1$. 
     This can be achieved for $\alpha_0>3$ and for a sufficiently large number $n$. 
\medskip

    \textbf{Case III} - $\dim V_x=2$. 
    
     Assume that, for each very general point $x\in X'$, we have  $\dim V_x=2$. Then we may spread these non-klt centers into a family by \cite[Lemma 3.2]{Todorov}. In explicit, there exists a smooth projective 3-fold $X''$ and the following diagram:
    $$
    \begin{tikzcd}
    X'' \arrow[r, "g"] \arrow[d, "h"] & X' \\
    B                                 &   
    \end{tikzcd}
    $$
    where
    \begin{enumerate}
        \item $h$ is a fibration onto a smooth curve $B$;
        \item $g$ is a generically finite morphism;
        \item for a general fiber $X_b''$ of $h$, there exists some $x\in X'$ such that $g(X_{b}'')$ is the minimal non-klt center of $(X',\Delta_x)$ at $x$.
    \end{enumerate}
    By Lemma \ref{pencil}, we get a linear pencil of surfaces $g(X_{b}'')$ on $X'$ since we have $q(X)=0$. Considering the rational map induced by this pencil, we may construct a fibration $X'\to \bP^1$ after passing to a higher model of $X'$ by  taking the Stein factorization and replacing $f$ and $\pi$ correspondingly. Note that we always have a rational pencil since $q(X')=0$ by the Leray spectral sequence (see \cite[Section 2.6]{EXPII}). Denote by $F$ a general fiber of this fibration, and it follows that $F$ is a higher model of some $V_x$. Since $K_{X'}\geq\frac{1}{\lambda}V_x$ and $F$ is movable, by considering the movable part of pluricanonical system, we naturally have $\pi^*(K_{X})\geq\frac{1}{\lambda} F$. By assumption, we still have $p_g(F)>0$. 
    
    Suppose $\Vol(F)\geq 4$. To apply Lemma \ref{5kfib}(5) to get the birationality of $|5K_{X'}|$,  we need $1/\lambda>1$. So it suffices to take such a number $\alpha_0>3$. Suppose $\Vol(F)\leq 3$.    Then,  by \cite[Lemma 2.6(1)]{Chen14}, we have $\pi^*(K_{X})\geq (4+\varepsilon) F$ for a small number $\varepsilon>0$ whenever $\Vol(X)>36 $. It follows from Lemma \ref{5kfib}(1) that $|5K_{X'}|$ gives a birational map. This deduces a contradiction which means that Case III does not happen. 
\end{proof}

\subsection{The case with $d_L=3$}\

    Denote by $J\in|S|_S|$ a general member. Note that $J$ is irreducible since $d_L>2$. First we need the following lemma:
    
\begin{lem}\label{d=3lem} 
Keep the setting in \ref{assumption}. Then 
$$(\pi^*(K_{X})|_S)^2\geq \sqrt{\Vol(X)\cdot (\pi^*(K_{X})\cdot S^2)}. $$
Furthermore assume $d_L=3$, then 
$$(\pi^*(K_{X})\cdot S^2)\geq\sqrt{2(\pi^*(K_{X})|_S)^2}\  \text{and}\ (\pi^*(K_{X})|_S)^2\geq 2^{\frac{1}{3}}(\Vol(X))^{\frac{2}{3}}. $$
\end{lem}

\begin{proof}
    Take a sufficiently large integer $l$ such that $|l\pi^*(K_X)|$ is base point free. Then a general member $T\in|l\pi^*(K_X)|$ is a smooth surface. By the Hodge index theorem, we have
    \begin{eqnarray*}
        (\pi^*(K_{X})|_S)^2&=&\frac{1}{l}(l\pi^*(K_{X})\cdot \pi^*(K_{X})\cdot S)\\
        &=&\frac{1}{l}(\pi^*(K_{X})|_T\cdot S|_T)\geq\frac{1}{l}\sqrt{\pi^*((K_{X})|_T)^2\cdot (S|_T)^2}\\
        &=&\sqrt{\Vol(X)\cdot (\pi^*(K_{X})\cdot S^2)}.
    \end{eqnarray*}
    Suppose $d_L=3$, then $|J|$ gives a generically finite map and $S$ is of general type. It follows that $J^2\geq 2$. Hence we have
    \begin{equation*}
        (\pi^*(K_{X})\cdot S^2)\geq\sqrt{(\pi^*(K_{X})|_S)^2\cdot (J)^2}\geq\sqrt{2(\pi^*(K_{X})|_S)^2}
    \end{equation*}
 and $(\pi^*(K_X)|_S)^2\geq 2^{\frac{1}{3}} (\Vol(X))^{\frac{2}{3}}$.   
\end{proof}

\begin{prop}\label{d=3} 
Let $m\geq 5$ be an integer and $X$ be a minimal projective $3$-fold of general type. Assume that $X$ satisfies the following conditions:
\begin{enumerate}
    \item $\Vol(X)>\alpha_0^3>(2\sqrt{6}/3)^3$ for some number $\alpha_0>0$;    
    \item there is a Weil divisor $L$ on $X$ such that $h^0(X,L)\geq 4$ and that 
$D=2K_X-L$ is pseudo-effective.
\end{enumerate}
If $d_L=3$ and $|mK_{X'}-M'|\neq \varnothing$, then $|mK_X|$ gives a birational map.
\end{prop}
\begin{proof} The argument of discussing $|5K_{X}|$ works for all other cases with $m>5$. Here we only provide the proof for the birationality of $|5K_{X}|$. 
    
By Tankeev's principle \cite[Section 2.1]{Che01}, it suffices to show that, for a general member $S\in|M'|$, the linear system $|5K_X'||_S$ gives a birational map. Now we fix a sufficiently large integer $n$ and a general effective $\bQ$-divisor 
    \begin{align*}
        E_n\sim_{\bQ}Z'+\pi^*D+\frac{1}{n}\pi^*(K_{X}).
    \end{align*}
    Notice that
    \begin{align*}
        |5K_{X'}|\lsgeq |K_{X'}+\ulcorner4\pi^*(K_{X})-E_n\urcorner|
    \end{align*} and
    \begin{align*}
        4\pi^*(K_{X})-S-E_n\sim_{\bQ}(2-\frac{1}{n})\pi^*(K_X)
    \end{align*} is nef and big. By Kawamata-Viehweg vanishing theorem \cite{KV,VV}, we have the surjective map 
 \begin{align*}
     H^0(X',K_{X'}+\ulcorner4\pi^*(K_{X})-E_n\urcorner)\to H^0(S,K_S+\ulcorner4\pi^*(K_{X})-S-E_n\urcorner|_{S}).
 \end{align*}
  Let $Q=(4\pi^*(K_{X})-S-E_n)|_S$, then
  \begin{align*}
      |K_S+\ulcorner4\pi^*(K_{X})-S-E_n\urcorner|_{S}|\lsgeq|K_S+\ulcorner Q\urcorner|.
  \end{align*}
  So we only need to show that $|K_S+\ulcorner Q\urcorner|$ induces a birational map. By Lemma \ref{d=3lem}, we have
  \begin{align*}
      Q^2\geq(2-\frac{1}{n})^2(\Vol(X))^{\frac{2}{3}}> (2-\frac{1}{n})^2\alpha_0^2  >8.
  \end{align*}
  By Masek-Langer's criterion \cite{Langer2001,Masek1999}, if $|K_S+\ulcorner Q\urcorner|$ does not give a birational map, then there exists a family of curves $C$ passing through two very general points of $S$ such that 
\begin{align}\label{masekpf2}
    (Q\cdot C)\leq\frac{4}{1+\sqrt{1-\frac{8}{Q^2}}}.
\end{align}
Since $d_L=3$, the linear system $|J|$ induces a generically finite map on $S$, and $|J|_C|$ induces a generically finite map on a general $C$. It follows that $h^0(J|_C)\geq 2$, and hence $\deg_C(J|_C)\geq 2$ by Clifford's theorem. Then
\begin{align*}
    (K_S\cdot C)=(K_{X'}|_S\cdot C)+(J\cdot C)>2
\end{align*}
since $K_{X'}$ is big. By the adjunction formula we have $g(C)\geq3$.

Let $\sigma:S\to S_0$ be the contraction onto the minimal model, and $\Tilde{C}=\sigma_*C$. Suppose $|\Tilde{C}|$ is free, then $(K_{S_0}\cdot \Tilde{C})=2g(\Tilde{C})-2\geq 4$.  Suppose otherwise $|\Tilde{C}|$ is not free, then the Hodge index theorem gives $(K_{S_0}\cdot \Tilde{C})\geq\roundup{\sqrt{\Vol(S)}}$. Notice that 
\begin{eqnarray*}
\Vol(S)&=&\Vol((K_{X'}+S)|_S)>\Vol(\pi^*(K_X)|_S+S|_S)\\
&=& (\pi^*(K_{X})|_S)^2+2(\pi^*(K_{X})\cdot S^2)+(S|_S)^2\\
&>& 2^{\frac{1}{3}}\alpha_0^2+2\sqrt{2^{\frac{4}{3}}\alpha_0^2}+2
\end{eqnarray*}
by Lemma \ref{d=3lem}. So in both cases we always have $(\sigma^*(K_{S_0})\cdot C)=(K_{S_0}\cdot \Tilde{C})\geq 4$.

By \cite[Theorem 2.4]{CZ16}, we have
\begin{align*}
    |l(K_{X'}+S)||_S=|lK_S|
\end{align*}
for any sufficiently large and divisible integer $l$. Since $|l\sigma^*(K_{S_0})|$ is base point free, we have $\Mov(l(K_{X'}+S))|_S\geq l\sigma^*(K_{S_0})$. Noting that
\begin{align*}
    (3l+1)\pi^*(K_{X})\geq \Mov((3l+1)K_{X'})\geq\Mov(l(K_{X'}+S)),
\end{align*} 
it follows that
\begin{align}\label{extension}
    \pi^*(K_X)|_S\sim_{\bQ}\frac{l}{3l+1}\sigma^*(K_{S_0})+H_S,
\end{align}
where $H_S$ is an effective $\bQ$-divisor on $S$. Hence
\begin{align*}
    (Q \cdot C)=(2-\frac{1}{n})(\pi^*(K_X)|_S\cdot C)\geq \frac{l(2n-1)}{n(3l+1)}(\sigma^*(K_{S_0})\cdot C)\geq \frac{4l(2n-1)}{n(3l+1)}.
\end{align*}
This contradicts (\ref{masekpf2}) when $l$ and $n$ are sufficiently large. So $|K_S+\ulcorner Q\urcorner|$ induces a birational map, which shows that $|5K_X|$ gives a birational map. 
\end{proof}

\section{The birationality of $\varphi_{5,X}$ (Part II: $d_L=2$)}

Keep the setting in \ref{assumption}. Assume $d_L=2$. We have an induced fibration $f:X'\lrw \Sigma$ from $|L|$. Pick a general fiber $C$ of $f$. We may write $S|_S\equiv aC$, where $C$ is a general fiber of $f'$ and, clearly, $a\geq h^0(L)-2$. Denote by $\xi:=(\pi^*(K_X)\cdot C)$ the canonical degree of $C$.

\subsection{The case with large canonical degree}\

\begin{prop}\label{d=2} 
Let $m\geq 5$ be an integer and $X$ be a minimal projective $3$-fold of general type. Assume that $X$ satisfies the following conditions:
\begin{enumerate}
    \item $\Vol(X)>\alpha_0^3>83$ for some number $\alpha_0>0$;
    \item there is a Weil divisor $L$ on $X$ such that $h^0(X,L)\geq 5$ and that 
$D=2K_X-L$ is pseudo-effective.
\end{enumerate}
Suppose that $d_L=2$, $\xi>76/71$, and that $|mK_{X'}-M'|\neq \varnothing$. Then $|mK_X|$ gives a birational map.

\end{prop}

\begin{proof} The argument for $|5K_{X}|$ clearly works for all cases with $m>5$. Here we only provide the proof for the birationality of $|5K_{X}|$. 

    Keeping symbols $Q$ and $E_n$,  and using similar arguments in the proof of Proposition  \ref{d=3}, we only need to show that $|K_S+\ulcorner Q\urcorner|$ induces a birational map. When $n$ is sufficiently large, applying Kawamata-Viehweg vanishing theorem as in the proof of Lemma \ref{5kfib}, one sees  that $|K_S+\ulcorner Q\urcorner|$ separates different generic irreducible  elements of $|C|$. So, by Tankeev's principle \cite[Section 2.1]{Che01}, it suffices to show  $|K_S+\ulcorner Q\urcorner||_C$ induces a birational map.

    Note that
    \begin{align*}
       Q-C-\frac{1}{a}E_n|_S\equiv\left(2-\frac{2}{a}-\frac{1}{n}(1+\frac{1}{a})\right)\pi^*(K_X)|_S
    \end{align*}
    is nef and big when $n$ is sufficiently large and $a>1$. By Kawamata-Viehweg vanishing theorem \cite{KV,VV}, we have the surjective map 
 \begin{align*}
     H^0(S,K_{S}+\ulcorner Q-\frac{1}{a}E_n|_S\urcorner)\to H^0(C,K_C+\ulcorner Q-C-\frac{1}{a}E_n|_S\urcorner|_{C}).
 \end{align*}
  Let $P=(Q-C-\frac{1}{a}E_n|_S)|_C$, then
  \begin{align*}
      |K_C+\ulcorner Q-C-\frac{1}{a}E_n|_S\urcorner|_{C}|\lsgeq|K_S+\ulcorner P\urcorner|,
  \end{align*}
  so we only need to show that $|K_C+\ulcorner P\urcorner|$ induces a birational map.

Note that
    \begin{align*}
        \deg(P)=\left(2-\frac{2}{a}-\frac{1}{n}(1+\frac{1}{a})\right)\xi
    \end{align*}
    So we have $\deg(P)>2$ whenever $\xi> 2$ and $n$ is sufficiently large. In this case, $|K_C+\ulcorner P\urcorner|$ induces a birational map and $\varphi_{5,X}$ is birational.

    When $76/71<\xi\leq 2$, we argue in another way to show $\varphi_{5,X}$ is birational. By \cite[Lemma 2.7]{Chen14} and Lemma \ref{d=3lem}, we have
    \begin{align*}
        \pi^*(K_X)|_S \geq \left(\frac{(\pi^*(K_X)|_S)^2}{2(\pi^*(K_X)|_S\cdot C)} - \varepsilon \right) C \geq \frac{1}{2}\sqrt{\frac{\alpha_0^3a}{\xi}}C
    \end{align*}
   where $\varepsilon$ is a sufficiently small positive number.  Take 
    \begin{align*}
        0\leq H_c\sim_{\bQ}\pi^*(K_X)|_S-\beta C, \quad \text{where}\ \beta=\frac{1}{2}\sqrt{\frac{\alpha_0^3a}{\xi}}
    \end{align*}
    Note that
    \begin{align*}
       Q-C-\frac{1}{\beta}H_c\equiv\left(2-\frac{1}{\beta}-\frac{1}{n}\right)\pi^*(K_X)|_S
    \end{align*}
    is nef and big when $n$ is sufficiently large and $\beta>1/2$. By Kawamata-Viehweg vanishing theorem \cite{KV,VV}, we have the surjective map 
 \begin{align*}
     H^0(S,K_{S}+\ulcorner Q-\frac{1}{\beta}H_c\urcorner)\to H^0(C,K_C+\ulcorner Q-C-\frac{1}{\beta}H_c\urcorner|_{C}).
 \end{align*}
  Let $P'=(Q-C-\frac{1}{\beta}H_c)|_C$, then
  \begin{align*}
      |K_C+\ulcorner Q-C-\frac{1}{\beta}H_c\urcorner|_{C}|\lsgeq|K_C+\ulcorner P'\urcorner|,
  \end{align*}
  so we only need to show that $|K_C+\ulcorner P'\urcorner|$ induces a birational map.
  Note that
    \begin{align}\label{degp'}
        \deg(P')=\left(2-\frac{1}{\beta}-\frac{1}{n}\right)\xi=\left(2-2\sqrt{\frac{\xi}{\alpha_0^3a}}-\frac{1}{n}\right)\xi
    \end{align}
    So we have $\deg(P')>2$ whenever $76/71<\xi\leq 2$, $\alpha_0^3>83$, and $n$ is sufficiently large. In this case, $|K_C+\ulcorner P'\urcorner|$ induces a birational map and $\varphi_{5,X}$ is birational.
\end{proof}

\begin{rem}
    In fact, under the same assumption as that of Proposition \ref{d=2}, we always have $\xi\geq \frac{1}{3} (\sigma^*(K_{S_0})\cdot C)\geq 4/3$ when $g(C)\geq 3$. So the difficulty necessarily arises when $g(C)=2$.
\end{rem}

\subsection{The case with small canonical degree}\

The difficulty arises while we are treating the case with a smaller canonical degree $\xi$. The key point is that we will show the existence of a class of $(1,2)$-surfaces.  

\begin{prop}\label{12fib} 
Let $m\geq 5$ be an integer and $X$ a minimal projective $3$-fold of general type. Assume that $X$ satisfies the following conditions:
\begin{enumerate}
    \item $\Vol(X)>\alpha_0^3>83$ for some number $\alpha_0>0$;
    \item $q(X)=h^1(X,\OX)=0$;    
    \item there is a Weil divisor $L$ on $X$ such that $h^0(X,L)\geq 5$ and that $D=2K_X-L$ is pseudo-effective.
\end{enumerate}
Suppose $d_L=2$, $\xi \leq 76/71$, $|mK_{X'}-M'|\neq \varnothing$ and that $\varphi_{m,X}$ is non-birational. Then $X'$ is birationally fibred by (1,2)-surfaces. 
\end{prop}

\begin{proof}
     The argument for $|5K_{X}|$ works for all cases with $m>5$. So it suffices to study the case with $m=5$.

    Applying Fujita's approximation (see \cite[Subsection 11.4]{La}, and also \cite{Takayama06}) and possibly replacing $X'$ by a higher birational model, we may write $K_{X'}\sim_{\bQ}A+E$, where $E$ is an effective $\bQ$-divisor and $A$ is an ample $\bQ$-divisor such that $0<\Vol(X)-\Vol(A)\ll1$. Pick very general points $x,y\in X'$. There exists an effective $\bQ$-divisor $\Delta_1\sim_{\bQ}t_1K_{X'}$ with $t_1<\frac{3\sqrt[3]{2}}{\sqrt[3]{\Vol(X) }}+\epsilon$, where
    $0<\epsilon\ll1$, such that $(X',\Delta_1)$ is log canonical but not klt at $x$, and that $(X',\Delta_1)$ is not klt at $y$. Recall our assumption for $\Vol(X)$, we may choose $t_1<3\sqrt[3]{2}/\alpha_0$. By \cite[Lemma 2.5]{HM06}, after modulo a small perturbation, we may also assume that there exists a unique non-klt center of $(X',\Delta_1)$ at $x$, say $V_1$. We next discuss case by case on the dimension of $V_1$.

\medskip

    \textbf{Case I} - $\dim V_1=0$. 
    
    In this case, $x$ is an isolated non-klt center of $(X',\Delta_1)$. By Nadel vanishing we see that $|(\lfloor t_1\rfloor+2)K_{X'}|$ separates $x$ and $y$. To assure the birationality of $|5K_{X'}|$, we need $t_1<4$, which can be achieved under our assumption. Hence Case I does not happen. 
\medskip

    \textbf{Case II} - $\dim V_1=1$. 
     
    We apply Takayama's induction \cite[Proposition 5.3; Lemma 5.8]{Takayama06} to conclude that there exists an effective $\bQ$-divisor $\Delta_2\sim_{\bQ}t_2K_{X'}$ with 
    \begin{align*}
        t_2<2+2t_1+\epsilon, 
    \end{align*}
   where $\epsilon>0$ can be sufficiently small, such that $x$ is an isolated non-klt center of $(X',\Delta_2)$, and that $(X',\Delta_2)$ is not klt at $y$.
   Then we reduce it to case I, and we need $t_2<4$ to assure the birationality of $|5K_{X'}|$, which can be also achieved under our assumption. Hence Case II does not happen. 
\medskip

    \textbf{Case III} - $\dim V_1=2$. 
    
     In this case, we want to show that $X'$ is birationally fibred by $(1,2)$-surfaces.

     Arguing as in the proof of Lemma \ref{tankeev1}, we obtain a smooth projective 3-fold $X''$ and the following diagram:
         $$
    \begin{tikzcd}
    X'' \arrow[r, "g"] \arrow[d, "h"] & X' \\
    B                                 &   
    \end{tikzcd}
    $$
    where
    \begin{enumerate}
        \item $h$ is a fibration onto a smooth curve $B$;
        \item $g$ is a generically finite morphism;
        \item for a general fiber $X_b''$ of $h$, there exists some $x\in X'$ and an effective $\bQ$-divisor $\Delta_1\sim_{\bQ}t_1K_{X'}$ such that $g(X_{b}'')$ is the minimal non-klt center of $(X',\Delta_1)$ at $x$.
    \end{enumerate}
    Again arguing as in the proof of Lemma \ref{tankeev1}, we may construct a fibration $X'\to \bP^1$ after passing to a higher model of $X'$ and replace $f$ and $\pi$ correspondingly. And a general fiber of this fibration $F$ satisfies $\pi^*(K_{X})\geq\frac{1}{t_1} F$. Since $\xi=(\pi^*(K_X)\cdot C)\leq 76/71$, and $C$ moves over the surface $\Sigma$, we have  
    \begin{align*}
        (F\cdot C)\leq t_1(\pi^*(K_X)\cdot C)<1,
    \end{align*}
    which implies that $(F\cdot C)=0$. Hence $C$ is contained in a fiber of $X'\to \bP^1$. In other words, $F$ is fibered by some genus $2$ curves in the same algebraic family of $C$. By \cite[Lemma 3.3]{CZuo08}, there exists an effective $\bQ$-divisor $H$ on $F$ such that 
 \begin{align}
     \pi^*(K_X)|_F\equiv\frac{1}{1+t_1}\sigma^*(K_{F_0})+H,
 \end{align}
 where $\sigma:F\to F_0$ is the contraction onto the minimal model. It follows that
    \begin{align*}
        (\sigma^*(K_{F_0})\cdot C)\leq(1+t_1)(\pi^*(K_X)|_F\cdot C)<2.
    \end{align*}
    Hence $F$ must be a (1,2)-surface by the surface theory  (see \cite[Lemma 2.4]{EXPIII} for a direct reference), which finishes the proof.
\end{proof}

\begin{prop}\label{5kbird2} 
Let $m\geq 5$ be an integer and $X$ be a minimal projective $3$-fold of general type. Assume that $X$ satisfies the following conditions:
\begin{enumerate}
    \item $\Vol(X)>\alpha_0^3>83$ for some number $\alpha_0>0$;
    \item $q(X)=h^1(X,\OX)=0$;
\item there is a Weil divisor $L$ on $X$ such that $h^0(X,L)\geq 5$ and that 
$D=2K_X-L$ is pseudo-effective.
\end{enumerate}
Suppose that $d_L=2$, and that $|mK_{X'}-M'|\neq \varnothing$. Then $|mK_X|$ gives a birational map.
\end{prop}

\begin{proof}
    By Proposition \ref{d=2} and Proposition \ref{12fib}, we may assume that $X'$ is birationally fibred by (1,2)-surfaces. Since $\Vol(X)>12$. Then we have $\pi^*(K_{X})\geq (4+\delta)F$ by \cite[Lemma 2.6(1)]{Chen14} for some small positive rational number $\delta$. The statement follows directly from Lemma \ref{5kfib}(1).
\end{proof}

\section{Proof of Theorem \ref{mainthm}}

There remains one exceptional case to consider, namely,  $\chi(\OX)=1$. This situation can be treated directly  since one has sufficiently large plurigenera.

\begin{prop}\label{chi=1}
    Let $m\geq 5$ be an integer and $X$ be a minimal projective $3$-fold of general type. Assume that $X$ satisfies the following conditions:
\begin{enumerate}
    \item $\Vol(X)>\alpha_0^3>83$ for some number $\alpha_0>0$;
    \item $q(X)=h^1(X,\OX)=0$;
    \item $\chi(\OX)=1$. 
\end{enumerate}
Then $|mK_X|$ gives a birational map.
\end{prop}

\begin{proof}
    By Reid's Riemann-Roch formula \cite{Rei87}, we have
    \begin{eqnarray*}
        P_2(X)&=&h^0(X,\OX(2K_X))=\frac{1}{2}K_X^3-3\chi(\OX)+l(2)>10,\ \text{and}\\
        P_3(X)&=&h^0(X,\OX(3K_X))=\frac{5}{2}K_X^3-5\chi(\OX)+l(3)>0,
    \end{eqnarray*}
    where $l(2)$ and $l(3)$ are non-negative numbers according to Reid. So we may choose $L=2K_X$. 
    
    Suppose $d_L=1$. Since $q(X)=0$, $|L|$ induces the fibration $f:X'\lrw \bP^1$. Denote by $F$ a general fiber of $f$. We have $M'\sim aF$, where $a\geq h^0(L)-1>9$. It follows that $\pi^*(K_X)\sim_{\bQ}(4+\delta)F+\Delta$, where $\delta$ is a sufficiently small positive rational number and $\Delta$ is an effective $\bQ$-divisor. Then the statement follows from Lemma \ref{5kfib}(1).

    Suppose $d_L\geq 2$. Since $P_3(X)>0$, we have $|mK_{X'}-M'|\neq \varnothing$. Then the statement follows from Proposition \ref{d=3} and Proposition \ref{5kbird2}.
    
\end{proof}

\begin{proof}[Proof of Theorem \ref{mainthm}]
     By main results of \cite{CCCJ,CCZ07} and Proposition \ref{chi=1}, we may assume $q(X)=0$,  $\chi(\OX) \neq 1$, and $X$ is non-Gorenstein. If, moreover, $h^2(\OX)\geq 13$, then we may choose $L$ to be the divisor in Lemma \ref{2forms}.  Otherwise, we have $h^2(\OX)\leq 12$. then we have $\chi(\OX)=1+h^2(\OX)-p_g(X)\leq 13$. By Reid's Riemann-Roch formula \cite{Rei87}, we have
    \begin{align*}
        P_2(X)=h^0(X,\OX(2K_X))=\frac{1}{2}K_X^3-3\chi(\OX)+l(2)>4,
    \end{align*}
    where $l(2)>0$ since $X$ is non-Gorenstein. We may choose $L=2K_X$. Then the statement follows from Propositions \ref{d=1}, \ref{tankeev1}, \ref{d=3}, \ref{5kbird2}. We are done. 
\end{proof}

As a direct application of Lemma \ref{5kfib}, we can also improve Theorem 1.1 of \cite{Chen14}.

\begin{cor}
    Let $X$ be a minimal projective $3$-fold of general type. Assume that $p_g(X)=2$ and $\Vol(X)>12$. Then $|5K_X|$ gives a birational map.
\end{cor}

\begin{proof}
    The canonical linear system $|K_X|$ induces a fibration $f:X'\lrw \bP^1$. Denote by $F$ a general fiber of $f$. Since $p_g(X)>0$, we have $p_g(F)>0$. We always have $\pi^*(K_X)\geq F$.
    
    Suppose $\Vol(F)=1$. Since $\Vol(X)>12$. We have $\pi^*(K_{X})\geq (4+\delta)F$ by \cite[Lemma 2.6(1)]{Chen14} for some small positive rational number $\delta$. The statement follows directly from Lemma \ref{5kfib}(1). The case with $2\leq\Vol(F)\leq4$ can also be shown in the similar way by Lemma \ref{5kfib} (3),(4),(5).

    Suppose $\Vol(F)\geq5$. Since $p_g(X)>0$, $|mK_{X'}|$ separates different general fibers of $f$. So the assumption that $p>1$ is unnecessary in this situation. Moreover, relations (\ref{2.4}) and (\ref{2.5}) are contradictory to each other when $p\geq1$ and $\Vol(F)\geq5$. This case is therefore concluded by the proof of Lemma \ref{5kfib}.
\end{proof}

Apart from some known examples with non-birational $5$-canonical maps found in Fletcher \cite{Fletcher}  and Chen-Jiang-Li \cite{CJL}, we have the following more examples.

\begin{exmp}\   Denote by $X_d$ the general homogeneous hypersurface $X_d$ of degree $d$ in weighted projective spaces. The following $3$-folds are all minimal with at worst canonical singularities: 
\begin{enumerate}
\item $X_{26}\subset \bP(1,2,3,5,13)$, $p_g=2$, $K^3=8/15$, $\varphi_{6}$ is non-birational.

\item $X_{34}\subset \bP(1,3,4,6,17)$, $p_g=2$, $K^3=3/4$, $\varphi_{5}$ is non-birational. 

\item $X_{36}\subset \bP(1,2,5,7,18)$,
$p_g=2$, $K^3=27/35$, $\varphi_{5}$ is non-birational. 

\item $X_{56}\subset \bP(1,3,8,11,28)$, $p_g=2$, $K^3=125/132$, $\varphi_{5}$ is non-birational.

\end{enumerate}
\end{exmp}

It seems to be very difficult to find more examples with larger canonical volume and with non-birational $5$-canonical maps. According to \cite{Chen03}, such examples necessarily have  the geometric genus $p_g(X)\leq 3$. So we naturally ask the following question:

\begin{op} Let $X$ be a minimal $3$-fold of general type. Are the following lower bound conditions for the canonical volume optimal? 
\begin{enumerate}
\item When $p_g(X)=3$ and $K_X^3>6$, $|5K_X|$ gives a birational map.
\item When $p_g(X)=2$ and $K_X^3>12$, $|5K_X|$ gives a birational map.
\end{enumerate}

\end{op}

\section*{\bf Acknowledgments}

 Chen is a member of the Key Laboratory of Mathematics for Nonlinear Sciences, Fudan University. 


\begin{thebibliography}{99}

\bibitem{Bombieri73} 
E.~Bombieri, 
{\it Canonical models of surfaces of general type}, 
Publ. Math. IHES 
{\bf 42} (1973), 171--219.


\bibitem{EXPII} J.~A.~Chen, M.~Chen, {\it Explicit birational geometry of $3$-folds of general type, II}, J. Differential Geom. {\bf 86} (2010), no.~2, 237--271. 

\bibitem{EXPIII} J.~A.~Chen, M.~Chen, {\it Explicit birational geometry of $3$-folds and $4$-folds of general type, III}, Compos. Math. {\bf 151} (2015), no. 6, 1041--1082. 

\bibitem{CCCJ}
J.  J.~Chen, J.  A.~Chen, M.~Chen, Z.~Jiang,
{\it On quint-canonical birationality of irregular threefolds}, 
Proc. London Math. Soc. 
{\bf 122} (2021), no.~2, 234--258.

\bibitem{Che01} M.~Chen, {\it Canonical stability in terms of singularity index for algebraic threefolds}, Math. Proc. Cambridge Philos. Soc. {\bf 131}(2001), no.~2, 241-264. 

\bibitem{Chen03} M.~Chen, {\it Canonical stability of $3$-folds of general type with $p_g\ge 3$}, Internat. J. Math. {\bf 14} (2003), no. 5, 515--528.

\bibitem{CCZ07} J.~Chen, M.~Chen, D.-Q.~Zhang, {\it The 5-canonical system on 3-folds of general type}, J. Reine Angew. Math. {\bf 603} (2007), 165--181.

\bibitem{CZhang08} M.~Chen, D.-Q.~Zhang, {\it Characterization of the 4-canonical birationality of algebraic threefolds}, Math. Z. {\bf 258} (2008), 565--585.

\bibitem{CZuo08} M.~Chen, K.~Zuo, {\it Complex projective $3$-fold with non-negative canonical Euler--Poincar\'e characteristic}, Comm. Anal. Geom. {\bf 16} (2008), no. 1, 159--182. 

\bibitem{Chen12} M.~Chen, {\it On an efficient induction step with $\text{Nklt}(X,D)$ - notes to Todorov}, Comm. Anal. Geom. {\bf 20} (2012), no.~4, 765--779.

\bibitem{Chen14} 
M.~Chen, 
{\it Some birationality criteria on 3-folds with $p_g>1$}, 
Sci. China Math. 
{\bf 57} (2014), no.~11, 2215--2234.

\bibitem{CZ16}
M.~Chen, Q.~Zhang,
{\it Characterization of the 4-canonical birationality of algebraic threefolds, II},
Math. Z.
{\bf 283} (2016), 659--677.

\bibitem{CJ17} M.~Chen, Z.~Jiang, {\it A reduction of canonical stability index of $4$ and $5$ dimensional projective varieties with large volume}, Ann. Inst. Fourier (Grenoble) {\bf 67} (2017), no. 5, 2043--2082. 

\bibitem{Chen18}
M.~Chen,
{\it On minimal 3-folds of general type with maximal pluricanonical section index},
Asian J. Math.
{\bf 22} (2018), no. 2, 257--268.

\bibitem{CH21} 
M.~Chen and Y.~Hu, 
{\it On an open problem of characterizing the birationality of $4K$}, 
Comm. Anal. Geom. 
{\bf 29} (2021), no.~7, 1545--1557.

\bibitem{CJ22} 
M.~Chen and Z.~Jiang, 
{\it On projective varieties of general type with many global $k$-forms}, Preprint, arXiv:2211.00926 (2023).

\bibitem{CJL} M. Chen, C. Jiang, B. Li,  {\it On minimal varieties growing from quasismooth weighted hypersurfaces}, J.  Diff. Geom. {\bf 127} (2024), No. 1, 35--76.

\bibitem{CL24}
M.~Chen, H.~Liu,
{\it A lifting principle for canonical stability indices of varieties of general type}, 
J. Reine Angew. Math. 
{\bf 816} (2024), 19--45.

\bibitem{DiBiagio12}
L.~Di~Biagio,
{\it Pluricanonical systems for 3-folds and 4-folds of general type},
Math. Proc. Cambridge Philos. Soc.
{\bf 152} (2012), no. 1, 9--34.

\bibitem{Fletcher} A. R. Iano-Fletcher, {\it 
Working with weighted complete intersections}, 
London Math. Soc. Lecture Note Ser., {\bf 281} 
Cambridge University Press, Cambridge, 2000, 101--173.

\bibitem{Har} R.~Hartshorne, {\it Algebraic Geometry}, Graduate Texts in Mathematics, {\bf 52}, Springer-Verlag, New York, 1977.

\bibitem{HM06} C.~D.~Hacon, J.~McKernan, {\it Boundedness of pluricanonical maps of varieties of general type}, Invent. Math. {\bf 166} (2006), 1--25.

\bibitem{KV} Y. Kawamata, {\em A generalization of Kodaira-Ramanujam's vanishing theorem}, Math. Ann. {\bf 261} (1982), 43--46.

\bibitem{Langer2001}
A.~Langer,
{\it Adjoint linear systems on normal log surfaces},
Compos. Math.
{\bf 129} (2001), 47--66.

\bibitem{La} R.~Lazarsfeld, {\it Positivity in Algebraic Geometry II: Positivity for Vector Bundles, and Multiplier Ideals}, Springer-Verlag, Berlin, 2004.

\bibitem{Masek1999}
V.~Ma{\c s}ek,
{\it Very ampleness of adjoint linear systems on smooth surfaces with boundary},
Nagoya Math. J.
{\bf 153} (1999), 1--29.

\bibitem{McKernan02}
J.~McKernan,
{\it Boundedness of log terminal Fano pairs of bounded index},
Preprint,
arXiv:math/0205214 (2002).

\bibitem{Rei87} M.~Reid, {\it Young person's guide to canonical singularities}, In: Algebraic geometry, Bowdoin, 1985 (Brunswick, Maine, 1985), pp. 345--414,
Proc. Sympos. Pure Math., {\bf 46}, Part 1, American Mathematical Society, Providence, RI, 1987.

\bibitem{Takayama06}
S.~Takayama,
{\it Pluricanonical systems on algebraic varieties of general type},
Invent. Math.
{\bf 165} (2006), 551--587.

\bibitem{Todorov} 
G.~T.~Todorov, 
{\it Pluricanonical maps for threefolds of general type}, 
Ann. Inst. Fourier (Grenoble)
{\bf 57} (2007), no.~4, 1315--1330.

\bibitem{Tsuji06}
H.~Tsuji,
{\it Pluricanonical systems of projective varieties of general type. I},
Osaka J. Math.
{\bf 43} (2006), 967--995.

\bibitem{VV} E. Viehweg, {\it Vanishing theorems}, J. Reine Angew. Math. {\bf 335} (1982), 1--8.

\bibitem{W25} P. Wang, {\it  
On pluricanonical boundedness of varieties of general type}, Preprint, arXiv:2508.21459. 


\end{thebibliography}
\end{document}